\documentclass[12pt]{article}
\usepackage{graphicx} 

\title{Support materials}
\author{IC}
\date{\today}

\usepackage[a4paper,margin=1.0in]{geometry}
\usepackage[colorlinks,citecolor=magenta,linkcolor=black]{hyperref}
\pdfpagewidth=\paperwidth \pdfpageheight=\paperheight
\usepackage{amsfonts,amssymb,amsthm,amsmath,eucal,tabu,url,nccmath,empheq}
\usepackage{systeme}
\usepackage{fancyvrb}
\usepackage{pgf}
 \usepackage{array}
 \usepackage{tikz-cd}
 \usepackage{pstricks}
 \usepackage{pstricks-add}
 \usepackage{pgf,tikz}
 \usetikzlibrary{automata}
 \usetikzlibrary{arrows}
 \usepackage{indentfirst}
\usepackage{tabularx} 
\usepackage{booktabs}
\usepackage{xcolor}
\usepackage{longtable}

\theoremstyle{plain}
\newtheorem{thm}{Theorem}[section]
\newtheorem{theorem}[thm]{Theorem}

\theoremstyle{definition}
\newtheorem{definition}[thm]{Definition}
\newtheorem{remark}[thm]{Remark}
\newtheorem{example}[thm]{Example}

\newtheorem{thevarthm}[thm]{\varthmname}

\newenvironment{varthm*}[1]{\trivlist\item[]{\bf #1.}\it}{\endtrivlist}

\renewcommand\geq{\geqslant}

\renewcommand\leq{\leqslant}

\newcommand\be{\begin{eqnarray*}}
\newcommand\ee{\end{eqnarray*}}

\newcommand\newop[2]{\def#1{\mathop{\rm #2}\nolimits}}
\newop\log{log}
\newop\ord{ord}
\newop\Gal{Gal}
\newop\SL{SL}
\newop\Bl{Bl}
\newop\mult{mult}
\newop\mass{mass}
\newop\div{div}
\newop\codim{codim}
\newop\sing{sing}
\newop\vdim{vdim}
\newop\edim{edim}
\newop\Ass{Ass}
\newop\size{size}
\newop\reg{reg}
\newop\satdeg{satdeg}
\newop\supp{supp}
\newop\Neg{Neg}
\newop\Nef{Nef}
\newop\Nefh{Nef_H}
\newop\Eff{Eff}
\newop\Zar{Zar}
\newop\MB{MB}
\newop\MBxC{MB\mathit{(x,C)}}
\newop\NnB{NnB}
\newop\Bigg{Big}
\newop\Effbar{\overline{\Eff}}

\def\keywordname{{\bfseries Keywords}}%
\def\keywords#1{\par\addvspace\medskipamount{\rightskip=0pt plus1cm
\def\and{\ifhmode\unskip\nobreak\fi\ $\cdot$
}\noindent\keywordname\enspace\ignorespaces#1\par}}
\def\subclassname{{\bfseries Mathematics Subject Classification
(2020)}\enspace}
\def\subclass#1{\par\addvspace\medskipamount{\rightskip=0pt plus1cm
\def\and{\ifhmode\unskip\nobreak\fi\ $\cdot$
}\noindent\subclassname\ignorespaces#1\par}}

\begin{document}
\title{On the existence of maximizing curves of odd degrees}
\author{Marek Janasz and Izabela Leśniak
}
\date{\today}
\maketitle

\thispagestyle{empty}
\begin{abstract}
In this paper we provide the non-existence criterion for the so-called maximizing curves of odd degrees. Furthermore, in the light of our criterion, we define a new class of plane curves that generalizes the notion of maximizing curves which we call as $M$-curves.
\end{abstract}
\subclass{14H50, 32S25, 14C20}

\section{Introduction}
In this paper we study the notion of maximizing curves of odd degree. Let us recall that these curves were recently defined by Dimca and Pokora \cite{DimPok2023} in the spirit of maximizing curves of even degree defined by Persson \cite{Persson}. It is worth noting here that maximizing curves of odd degree are not directly related to the geometry of algebraic surfaces, unlike maximizing curves of even degree, but these curves enjoy certain extreme properties, such as being the so-called maximal Tjurina curves. Moreover, maximizing curves, both of even and odd degrees, are free curves, which makes them even more interesting from the perspective of questions bridging the combinatorics and homological properties of plane curves. The main goal of this paper is to approach a very natural problem of constructing maximizing plane curves of odd degrees, since these curves are extremely rare. To be more precise, we do not know \text{any} example of a maximizing curve of odd degree $\geq 11$, which is very surprising to us. Probably the main reason why it is notoriously difficult to construct such curves is the fact that maximizing curves admit (only) ${\rm ADE}$ singularities. Our main result in this paper tells us that if $C$ is a reduced plane curve of odd degree $2m+1 \geq 7$ admitting $A_{k}$ singularities with $k\leq 4m$, $D_{l}$ singularities with $l \leq 2m+1$ and $E_{6}$ singularities, then $C$ cannot be maximizing - see Theorem \ref{main} for details. This result directly explains by itself why it is so difficult to find maximizing curves of odd degrees. As a remedy for this problem, we propose a new class of reduced plane curves, which we will call $M$-curves. Roughly speaking, $M$-curves are reduced plane curves that admit ${\rm ADE}$ and simple elliptic singularities that are free, and have, in a strict sense, extreme values of the total Tjurina numbers. We provide a criterion for deciding whether a given curve $C$ is an $M$-curve, and present some examples to justify our hope that $M$-curves might be more accessible (constructionally) than maximizing curves.

Let us present an outline of our paper. In Section $2$ we provide all necessary definitions and results about maximizing curves of odd degree. In Section $3$ we give our main result, which is a handy criterion for the non-existence of some maximizing curves of odd degree. In Section $4$, we define $M$-curves and we describe their basic properties from the perspective of their total weak combinatorics. 

We work over the complex numbers and all necessary symbolic computations are performed using \verb}SINGULAR} \cite{Singular}.
\section{Preliminaries}

Let $S := \mathbb{C}[x,y,z]$ be the graded ring of polynomials with the complex coefficients. For a homogeneous polynomial $f \in S$ let us denote by $J_{f}$ the Jacobian ideal associated with $f$, namely this is the ideal of the form $J_{f} = \langle \partial_{x}\, f, \partial_{y} \, f, \partial_{z} \, f \rangle$. Since we are going to work with singularities of plane curves, we need to recall some basics and introduce the notation.

\begin{definition}
Let $p$ be an isolated singularity of a polynomial $f\in \mathbb{C}[x,y]$. Since we can change the local coordinates, assume that $p=(0,0)$.

The number 
$$\mu_{p}=\dim_\mathbb{C}\left(\mathbb{C}\{x,y\} /\bigg\langle \frac{\partial f}{\partial x},\frac{\partial f}{\partial y} \bigg\rangle\right)$$
is called the Milnor number of $f$ at $p$.

The number
$$\tau_{p}=\dim_\mathbb{C}\left(\mathbb{C}\{x,y\} /\bigg\langle f,\frac{\partial f}{\partial x},\frac{\partial f}{\partial y}\bigg\rangle \right)$$
is called the Tjurina number of $f$ at $p$.
\end{definition}

For a projective situation, with a point $p\in \mathbb{P}^{2}_{\mathbb{C}}$ and a homogeneous polynomial $f\in \mathbb{C}[x,y,z]$, we can take local affine coordinates such that $p=(0 : 0 : 1)$ and then the dehomogenization of $f$.\\
Finally, the total Tjurina number of a given reduced curve $C \subset \mathbb{P}^{2}_{\mathbb{C}}$ is defined as
$${\rm deg}(J_{f}) = \tau(C) = \sum_{p \in {\rm Sing}(C)} \tau_{p}.$$ 
Moreover, if $C = \{ f=0 \}$ is a reduced plane curve with only quasi-homogeneous singularities (i.e., there exists a holomorphic change of variables so that the defining equation becomes weighted homogeneous), then for every singular point $p \in {\rm Sing}(C)$ one has $\tau_{p} = \mu_{p}$, and hence $$\tau(C) = \sum_{p \in {\rm Sing}(C)} \tau_{p} = \sum_{p \in {\rm Sing}(C)} \mu_{p} = \mu(C),$$
which means that the total Tjurina number of $C$ is equal to the total Milnor number of $C$.

We will need the following important invariant that is attached to the syzygies of $J_{f}$.
\begin{definition}
Consider the graded $S$-module of Jacobian syzygies of $f$, namely $$AR(f)=\{(a,b,c)\in S^3 : af_x+bf_y+cf_z=0\}.$$
The minimal degree of non-trivial Jacobian relations for $f$ is defined as
$${\rm mdr}(f):=\min\{r : AR(f)_r\neq (0)\}.$$ 
\end{definition}

Now we can introduce the notion of a free curve that was introduced in \cite{Saito}.
\begin{definition}
We say that a reduced curve $C = \{f = 0\}$ in the complex projective plane is \textbf{free} if the Jacobian ideal $J_{f}$ associated with $f$ is saturated with respect to the irrelevant ideal $\mathfrak{m} = \langle x, y, z \rangle$. 
\end{definition}
It is very difficult to check the freeness of curves using the above definition. However, we have the following nice criterion that focuses on the total Tjurina number and the minimal degree of the Jacobian relations \cite{duP}.
\begin{theorem}[Freeness criterion]
Let $C = \{ f=0 \}$ be a reduced curve in $\mathbb{P}^{2}_{\mathbb{C}}$ of degree $d$. Then the curve $C$  is free if and only if
\begin{equation}
\label{duPles}
(d-1)^{2} - r(d-r-1) = \tau(C).
\end{equation}
\end{theorem}

Next, let us recall the classification of ${\rm ADE}$ singularities for curves by presenting their local normal forms captured from \cite{arnold}, namely
\begin{center}
\begin{tabular}{ll}
$A_{k}$ with $k\geq 1$ & $: \, x^{2}+y^{k+1}  = 0$, \\
$D_{k}$ with $k\geq 4$ & $: \, y^{2}x + x^{k-1}  = 0$,\\
$E_{6}$ & $: \, x^{3} + y^{4} = 0$, \\
$E_{7}$ & $: \, x^{3} + xy^{3} = 0$, \\
$E_{8}$ & $:\, x^{3}+y^{5} = 0$.
\end{tabular}
\end{center}
All reduced plane curves that admit only ${\rm ADE}$ singularities will be called \textbf{simply singular curves}.
In this paper we want to understand the existence of maximizing curves in odd degrees. Now we define our main object of studies.
\begin{definition}
Let $C = \{ f=0 \}$ be a reduced simply singular curve in $\mathbb{P}^{2}_{\mathbb{C}}$ of \textbf{odd degree} $n=2m+1\geq 5$. We say that $C$ is a \textit{maximizing} curve if 
$$\tau(C) = 3m^{2}+1.$$
\end{definition}
\begin{remark}
It is worth recalling that the property of being a maximizing curve $C = \{f=0\}$ of degree $2m+1$ is equivalent to the fact that this curve is free. Hence maximizing curves are free curves. Moreover, the maximality of $C = \{f=0\}$ implies that ${\rm mdr}(f)=m-1$.
\end{remark}

\section{The non-existence criterion for maximizing curves of odd degrees} 
Here is our main result of the paper.
\begin{theorem}
\label{main}
Let $C = \{ f=0 \}$ be a reduced simply singular curve in $\mathbb{P}^{2}_{\mathbb{C}}$ of \textbf{odd degree} $d=2m+1$ for $m \geq 3$ such that $C$ admits
\begin{itemize}
    \item singularities of type $A_k$ for $k \leq 4m$,
    \item singularities of type $D_l$ for $l \leq 2m+1$, and
    \item singularities of type $E_6$.
\end{itemize}
Then $C$ \textbf{cannot be} maximizing. \\
Furthermore, if $m \geq 5$, then we can add to the above list $E_7$ singularities. Finally, if $m \geq 8$, then we can add to the above list both $E_7$ and $E_8$ singularities.
\end{theorem}

\begin{proof}
    Let $C = \{f=0\}$ be a maximizing curve of degree $d=2m+1 \geq 7$ having $\tau(C)=3m^2+1$. Then $C$ is free with ${\rm mdr}(f) = m-1$. Now we are going to use the following result due to Dimca and Sernesi \cite[Theorem 1.2]{DimSer}. Since $C$ admits only ${\rm ADE}$ singularities, which are obviously quasi-homogeneous, then one has 
    $${\rm mdr}(C) \geq \alpha_{C}\cdot {\rm deg}(C) - 2,$$
where $\alpha_{C}$ denotes the Arnold exponent of $C$ that is defined as
$$\alpha_{C} = \min_{p \in {\rm Sing}(C)} \bigg\{{\rm lct}_{p}(C)\bigg\}.$$
Here by ${\rm lct}_{p}(C)$ we mean the log canonical threshold of $p \in {\rm Sing}(C)$.
Since ${\rm mdr}(f) = m-1$, we get
$$m-1 = {\rm mdr}(f) \geq \alpha_{C}\cdot(2m+1)-2,$$
and hence
\begin{equation}
\alpha_{C} \leq \frac{m+1}{2m+1}.    
\end{equation}
Summing up this part of the proof, if $C$ is maximizing of odd degree $2m+1$, then $\alpha_{C} \leq \frac{m+1}{2m+1}$. We can rephrase this condition in a way that it suits better for our purposes, namely if $C$ is a simply singular curve of degree $2m+1$ such that $\alpha_{C} > \frac{m+1}{2m+1}$, then $C$ cannot be maximizing. Now we are going determine, for a simply singular curve $C$ of a given degree $d = 2m+1$, these ${\rm ADE}$ singularities for which $C$ is not maximizing. In order to do so, we need to recall the following (see for instance \cite{PN}):
    \begin{enumerate}
        \item[a)] If $p \in {\rm Sing}(C)$ is a singularity of type $A_k$ with $k \geq 1$, then ${\rm lct}_{p}(C)=\frac{k+3}{2(k+1)}$.
        \item[b)] If $p \in {\rm Sing}(C)$ is a singularity of type $D_k$ with $k \geq 4$ then ${\rm lct}_{p}(C)=\frac{k}{2(k-1)}$.
        \item[c)] If $p \in {\rm Sing}(C)$ is a singularity of type $E_6$ then ${\rm lct}_{p}(C)=\frac{7}{12}$.
        \item[d)] If $p \in {\rm Sing}(C)$ is a singularity of type $E_7$ then ${\rm lct}_{p}(C)=\frac{5}{9}$.
        \item[e)] If $p \in {\rm Sing}(C)$ is a singularity of type $E_8$ then ${\rm lct}_{p}(C)=\frac{8}{15}$.
    \end{enumerate}      
Assume that $C$ of a fixed degree $d=2m+1 \geq 7$ \textbf{is not maximizing}. Then focusing on $A_{k}$ singularities $p \in {\rm Sing}(C)$ the following condition holds
    $${\rm lct}_{p}(C) = \frac{k+3}{2(k+1)} >  \frac{m+1}{2m+1},$$
and simple computations reveal that $k \leq 4m$. \\

Next, let us focus on $D_{l}$ singularities $p \in {\rm Sing}(C)$, then we have
    $${\rm lct}_{p}(C) = \frac{l}{2(l-1)} > \frac{m+1}{2m+1},$$
and after simple manipulations we get $l \leq 2m+1$. \\

Assume now that $p \in {\rm Sing}(C)$ is an $E_{6}$ singularity, then
$${\rm lct}_{p}(C) =\frac{7}{12} > \frac{m+1}{2m+1},$$
and since $m\geq 3$, this inequality is always satisfied.

Summing up, we proved that if $C$ is a simply singular curve of degree $d=2m+1 \geq 7$ admitting $A_{k}$ singularities with $k\leq 4m$, $D_{l}$ singularities with $l \leq 2m+1$ and $E_{6}$ singularities, then $C$ cannot be maximizing.

Assume now that our simply singular curve $C$ is of degree $d=2m+1 \geq 11$ and let $p \in {\rm Sing}(C)$ be a singularity of type $E_{7}$. Then we have
    $${\rm lct}_{p}(C) = \frac{5}{9} > \frac{m+1}{2m+1},$$
and by the assumption that $m \geq 5$ this inequality always holds. Hence if $C$ is a simply singular curve of degree $d=2m+1 \geq 11$ admitting $A_{k}$ singularities with $k\leq 4m$, $D_{l}$ singularities with $l \leq 2m+1$, $E_{6}$ and $E_{7}$ singularities, then $C$ cannot be maximizing.

Finally, let $C$ be a simply singular curve of degree $d=2m+1 \geq 17$ and let $p \in {\rm Sing}(C)$ be a singularity of type $E_{8}$. Then we have
    $${\rm lct}_{p}(C)  = \frac{8}{15} > \frac{m+1}{2m+1}$$
and since $m \geq 8$, this inequality is always satisfied.  Hence if $C$ is a simply singular curve of degree $d=2m+1 \geq 17$ admitting $A_{k}$ singularities with $k\leq 4m$, $D_{l}$ singularities with $l \leq 2m+1$, $E_{6}$, $E_{7}$ and $E_{8}$ singularities, then $C$ cannot be maximizing. This completes the proof. 
\end{proof}
As we mentioned in Introduction, we are not aware of so many examples of maximizing curves of odd degree, and the above results show that it will be a really challenging problem to find new examples maximizing curves of odd degree. To justify our prediction, the example below shows that our criterion is optimal in degree $d=9$.
\begin{example}[The only known maximizing curve of degree $9$, {\cite[Example 6.4]{DIPS}}]
Let us consider the following curve
$$C = \{f=(x^2+y^2+z^2)^3-27x^2y^2z^2=0\},$$
which is the dual of the quartic with $3$ nodes $C' : \{x^{2}y^{2}+ y^{2}z^{2} + x^{2}z^{2}=0\}$.
This sextic curves has six simple cusps $A_2$ at the points $(0:1:\pm i)$, $(1:0: \pm i )$ and $ (1:\pm i: 0)$, where $i^2=-1$, and four nodes $A_1$ at the points $(1:\pm 1: \pm 1)$.  One can check, using \verb}SINGULAR}, that $C$ is not free. 
If we add the lines $L_1:x=0$, $L_2:y=0$ and $L_3:z=0$, the resulting curve
$$C''=C \cup L_{1} \cup L_{2} \cup L_3 = \{f'=xyz\left( (x^2+y^2+z^2)^3-27x^2y^2z^2 \right)=0\}$$
has six singularities $E_7$ at the points $(0:1:\pm i)$, $(1:0: \pm i )$ and $ (1:\pm i: 0)$, and seven nodes $A_1$ at the points
$(1:\pm 1: \pm 1)$, $(1:0:0)$, $(0:1:0)$ and $(0:0:1)$.
It follows that
$$\tau(C'')=49 = 3\cdot 4^{2} + 1,$$
and this equality implies that $C''$ is maximizing of degree $9$. 
\end{example}
\begin{remark}
As we saw above, simply singular plane curves $C$ of degree $9$ can be maximizing if they just admit $A_{1}$ and $E_{7}$ singularities, and this is also visible on the level of our criterion.
\end{remark}
\section{{\it M}-curves}
In the previous section we saw that it is a very difficult problem to construct maximizing curves of odd degree. This is very disappointing, but here we want to present a remedy by introducing slightly general notion (compared with maximizing curves) of $M$-curves. In order to define such curves,  we need to introduce simple elliptic singularities for plane curves, which are the following:
\begin{center}
\begin{tabular}{ll}
$X_{9}$ with $a\neq 4$ & $: \, x^{4} + y^{4} + ax^{2}y^{2} = 0$, \\
$T_{2,3,6}$ with $4a^{3}+27 \neq 0$ & $: \, x^{3} + y^{6} + ax^{2}y^{2} = 0$.
\end{tabular}
\end{center}
It is well-known that simple elliptic singularites (${\rm SE}$ singularities for short) are quasi-homogeneous and they have the same log canonical threshold equal to $\frac{1}{2}$. 

Let $C = \{f=0\} \subset \mathbb{P}^{2}_{\mathbb{C}}$ be a reduced plane curve of even degree $d=2m\geq 4$ admitting ${\rm ADE}$ and ${\rm SE}$ singularities, \textbf{meaning here that $C$ admits at least one $SE$ singularity of any type}.
Then $\alpha_{C}=\frac{1}{2}$, and by a result due to Dimca and Sernesi \cite[Theorem 1.2]{DimSer} we have
$${\rm mdr}(f) \geq \frac{1}{2}\cdot d - 2 =m-2.$$
It means for this class of plane curves the minimal possible value of ${\rm mdr}$ is equal to $m-2$. This observation leads us to the following crucial definition.
\begin{definition}
Let $C = \{f=0\} \subset \mathbb{P}^{2}_{\mathbb{C}}$ be a reduced plane curve of even degree $d=2m\geq 4$ admitting ${\rm ADE}$ and ${\rm SE}$ singularities, then $C$ is an $M$-curve if $C$ is free with ${\rm mdr}(f)=m-2$.
\end{definition}
The main result devoted to $M$-curves of even degree can be formulated as follows. Before that happen, we need the following technical result due to du-Plessis and Wall \cite{duP}.

\begin{theorem}
\label{thm1}
Let $C = \{f=0\}$ be a reduced plane curve of degree $d$ and let $r = {\rm mdr}(f)$.
Then the following two cases hold.
\begin{enumerate}
\item[a)] If $r < d/2$, then $\tau(C) \leq \tau_{max}(d,r)= (d-1)(d-r-1)+r^2$ and the equality holds if and only if the curve $C$ is free.
\item[b)] If $d/2 \leq r \leq d-1$, then
$\tau(C) \leq \tau_{max}(d,r)$,
where, in this case, we set
$$\tau_{max}(d,r)=(d-1)(d-r-1)+r^2- \binom{2r-d+2}{2}.$$
\end{enumerate}
\end{theorem}
Here is our main result in this section.
\begin{theorem}
\label{Mce}
Let $C = \{f=0\} \subset \mathbb{P}^{2}_{\mathbb{C}}$ be a reduced plane curve of even degree $d=2m\geq 4$ admitting ${\rm ADE}$ and ${\rm SE}$ singularities. Then $C$ is an $M$-curve if and only if $\tau(C) = 3m^{2}-3m+3$.
\end{theorem}
\begin{proof}
Assume that $C$ is an $M$-curve of degree $d=2m$ and $r = {\rm mdr}(f) = m-2$. Then, by the above mentioned result due to du-Plessis and Wall, the following equality holds:
\begin{equation}
    \tau(C) = (d-1)(d-r-1)+r^2 = (2m-1)(2m-m+2-1)+(m-2)^2 = 3m^{2} - 3m + 3,
\end{equation}
and this completes the proof for the first implication. For the revers implication, let us assume that $\tau(C) = 3m^{2}-3m+3$. We are going to show that $C$ is free. Since $r = {\rm mdr}(f) \geq m-2$, by Theorem \ref{thm1} we see that
\begin{equation}
\tau(C) \leq \tau_{max}(2m,m-2) = 3m^{2}-3m+3,
\end{equation}
and this follows from the fact that the function $\tau_{max}(d,r)$ is strictly decreasing as a function with respect to $r$ on the interval $[0,2m-1]$. This gives us that,
$$\tau(C) = \tau_{max}(2m,m-2),$$
which implies that $r=m-2$ and $C$ has to be free by Theorem \ref{thm1} a).
\end{proof}
Let us now present some examples of $M$-curves in even degrees. 

\begin{example}(The Hesse arrangement)
\begin{align*}
\mathcal{H} = \{f=xyz(x+y+z)(x+y+ez)&(x+y+e^2 z)(x+ey+z)(x+e^2 y+z)(ex+y+z)\\& (e^2 x+y+z)(ex+e^2 y+z)(e^2 x+ey+z) = 0\},
\end{align*}
where $e^2 + e + 1=0$. As we remember, the Hesse arrangement consists of $12$ lines and it has $12$ double and $9$ quadruple points, so $9$ singularities of type $X_{9}$. We know that $\mathcal{H}$ is free by a general theory of reflection arrangements, and ${\rm mdr}(f)=4$, $\tau(\mathcal{H}) = 93$, hence $\mathcal{H}$ is an example of an even degree $M$-curve.
\end{example}
The forthcoming examples come from a recent paper by Pokora \cite{PP2024}, where he constructed new examples of conic-line arrangements with ordinary quasi-homogeneous singularities, so arrangements admitting $A_{1}$, $D_{4}$ and $X_{9}$ singularities. For the simplicity, we denote by $n_{2}$ the number of $A_{1}$ singularities, $n_{3}$ the number of $D_{4}$ singularities, and by $n_{4}$ the number of $X_{9}$ singularities.

\begin{example}[$M$-curve of degree $8$]
    $$C\mathcal{L} = \{f=xy(x^2+y^2-z^2)(y+x-z)(y-x-z)(y+x+z)(y-x+z)=0\}.$$
    We can check that for this arrangement the weak combinatorics has the form $(n_{2},n_{3},n_{4}) = (3,0,4)$, which implies that $\tau(C\mathcal{L} ) =39$, and ${\rm mdr}(f)=2$. So this is an example of an even degree $M$ curve.
\end{example}

\begin{example}[$M$-curve of degree $12$]

\begin{multline*}
C\mathcal{L} = \{f= xyz(x-z)(x+z)(y+z)(y-z)(y-x-z)(y-x+z)(y-x) \\ (-x^{2}+xy-y^{2}+z^{2}) = 0\}.
\end{multline*}
This arrangement has the weak combinatorics $(n_{2},n_{3},n_{4})=(8,1,9)$, which implies that $\tau(\mathcal{CL}_{16})=93$ and ${\rm mdr}(Q_{16})=4$. So this is another example of an even degree $M$ curve.
\end{example}

Let us now focus on curves of odd degree. Let $C = \{f=0\} \subset \mathbb{P}^{2}_{\mathbb{C}}$ be a reduced plane curve of even degree $d=2m+1\geq 5$ admitting ${\rm ADE}$ and ${\rm SE}$ singularities, meaning here that $C$ admits at least one $SE$ singularity of any type.
Then again $\alpha_{C}=\frac{1}{2}$, and by the aforementioned result due to Dimca and Sernesi \cite[Theorem 1.2]{DimSer} we have
$${\rm mdr}(f) \geq \frac{1}{2}\cdot d - 2 = \frac{1}{2}(2m+1) - 2 \geq m-1.$$
It means for this class of plane curves the minimal possible value of ${\rm mdr}$ is equal to $m-1$. This observation leads us to the following crucial definition.
\begin{definition}
Let $C = \{f=0\} \subset \mathbb{P}^{2}_{\mathbb{C}}$ be a reduced plane curve of odd degree $d=2m+1\geq 5$ admitting ${\rm ADE}$ and ${\rm SE}$ singularities, then $C$ is an $M$-curve if $C$ is free with ${\rm mdr}(f)=m-1$.
\end{definition}
In the same sprit as for $M$-curves of even degrees, we have the following characterization.
\begin{theorem}
Let $C = \{f=0\} \subset \mathbb{P}^{2}_{\mathbb{C}}$ be a reduced plane curve of odd degree, $d=2m+1\geq 5$ admitting ${\rm ADE}$ and ${\rm SE}$ singularities. Then $C$ is an $M$-curve if and only if $\tau(C) = 3m^{2}+1$.
\end{theorem}
\begin{proof}
The proof goes analogously as for Theorem \ref{Mce}.
\end{proof}
Now we would like to present an example of $M$-curve of degree $11$ that was presented in \cite{PP2024}.
\begin{example}[$M$-curve of degree $11$]
Let us consider the following conic-line arrangement
\begin{align*}
C\mathcal{L}  = \{f=xy(x-z)(x+z)(y-z)(y+z)(y-x-z)&(y-x+z)(y-x)\\&(-x^2+xy-y^2+z^2)=0\}.      
\end{align*}
This arrangement has the following singularities
$$(n_{2},n_{3},n_{4}) = (6,4,6),$$
giving $\tau(C\mathcal{L}) = 76$, and ${\rm mdr}(f)=4$, hence $C\mathcal{L}$ is an example of an odd degree $M$-curve.
\end{example}

Finishing this note, we would like to focus our attention on $M$-line arrangements, i.e., line arrangements that are $M$-curves. Our goal now is to provide an effective combinatorial description of such $M$-line arrangements.
\begin{theorem}
\label{cm}
Let $\mathcal{L} = \{f=0\}$ be an $M$ arrangement of $d\geq 5$ lines. Then
\begin{enumerate}
\item[a)] If $d=2m+1 \geq 5$, then 
\begin{equation}
\label{cm1}
    n_{2} + 2n_{3} + 3n_{4} = m^{2}+2m-1.
\end{equation}
\item[b)] If $d=2m \geq 6$, then 
\begin{equation}
    n_{2} + 2n_{3} + 3n_{4} = m^{2} + m-3.
\end{equation}
\end{enumerate}
\end{theorem}
\begin{proof}
Since $\mathcal{L}$ is free, then the following holds:
\begin{equation}
\tau(\mathcal{L}) = (d-1)^{2} - d_{1}(d-d_{1}-1),
\end{equation}
where ${\rm mdr}(f) = d_{1}$. Observe that
$$\tau(\mathcal{L}) = \sum_{k \geq 2}(k-1)^{2}t_{k} = d^{2}-d - \sum_{k \geq 2} (k-1)t_{k},$$
where the last equality follows from the naive combinatorial count for line arrangements, namely
$$d^{2}-d = \sum_{k\geq 2}(k^{2}-k)t_{k}.$$ 
This gives us
$$\tau(\mathcal{L}) = d^{2} - d - \sum_{k\geq 2}(k-1)t_{k} = d^{2}-2d+1 -d_{1}(d-d_{1}-1) = d^{2}-2d+1 - d_{1}d_{2},$$
where the last equality follows from the freeness of $\mathcal{L}$, namely one has $d_{2} = d - d_{1}-1$.
After simple manipulations, we get
$$\sum_{k\geq 2}(k-1)t_{k} = d_{1}d_{2} + d_{1} + d_{2}.$$
Assume that $d=2m+1\geq 1$. Then $d_{1} = m-1$ and $d_{2}=m+1$, we get
$$n_{2} + 2n_{3} + 3n_{4} = (m-1)(m+1) + (m-1)+(m+1) = m^{2}+2m-1.$$
Finally, we assume that $d=2m\geq 6$. Then $d_{1} = m-2$ and $d_{2} = m+1$, and we obtain
$$n_{2} +2n_{3} + 3n_{4} = (m-2)(m+1) + (m-2)+(m+1) = m^{2}+m-3.$$
\end{proof}

Our Theorem \ref{cm}, even if this is just a sufficient condition, is a very easy-to-use tool that allows us to detect some interesting examples of $M$-line arrangements among sporadic simplicial line arrangements presented in \cite{Cuntz}.

\begin{example}
Let us consider the simplicial line arrangement $\mathcal{A}(9,1)$ which has $d=9$ lines, $n_{2}=6$, $n_{3}=4$, and $n_{4}=3$. One can check that \eqref{cm1} is satisfied, so this arrangement can be an $M$-curve. Using \verb}SINGULAR} we can check that ${\rm mdr} = 3$, and $\tau(\mathcal{L}) = 49  = 3m^{2}+1$, hence $\mathcal{A}(9,1)$ is an $M$-curve.

\end{example}
\begin{example}
Let us consider the simplicial line arrangement $\mathcal{A}(13,2)$ which has $d=13$ lines, $n_{2}=12$, $n_{3}=4$, and $n_{4}=9$. One can check that \eqref{cm1} is satisfied, so this arrangement can be an $M$-curve. Using \verb}SINGULAR} we can check that ${\rm mdr} = 5$, and $\tau(\mathcal{L}) = 109 = 3m^{2}+1$, hence $\mathcal{A}(13,2)$ is an $M$-curve.
\end{example}

Concluding our discussions in this part, we look at the Klein arrangement of lines.
\begin{example}
Let us recall that the Klein arrangement $\mathcal{K}$ of $d=21$ lines has exactly $n_{3}=28$ and $n_{4}=21$, and we can easily check that \eqref{cm1} is satisfied. We know that $\mathcal{K}$ is a free arrangement with $d_{1}=9$ and $d_{2}=11$, hence $\mathcal{K}$ is an $M$-line arrangement.
\end{example}

\section*{Acknowledgement}
We would like to thank Piotr Pokora for his guidance and for sharing his thoughts on how to generalize the notion of maximizing curves.

Marek Janasz is supported by the National Science Centre (Poland) Sonata Bis Grant \textbf{2023/50/E/ST1/00025}. For the purpose of Open Access, the authors have applied a CC-BY public copyright license to any Author Accepted Manuscript (AAM) version arising from this submission.

\vskip 0.5 cm
\bigskip
\noindent
Marek Janasz,\\
Department of Mathematics,\\
University of the National Education Commission Krakow,
Podchor\c a\.zych 2,
PL-30-084 Krak\'ow, Poland. \\
\nopagebreak
\textit{E-mail address:} \texttt{marek.janasz@uken.krakow.pl}
\bigskip

\noindent
Izabela Leśniak,\\
Department of Mathematics,\\
University of the National Education Commission Krakow,
Podchor\c a\.zych 2,
PL-30-084 Krak\'ow, Poland. \\
\nopagebreak
\textit{E-mail address:} \texttt{izabela.lesniak@uken.krakow.pl}

\begin{thebibliography}{000}

\bibitem{arnold}
V. I. Arnold, Local normal forms of functions. \textit{Invent. Math.} \textbf{35}: 87 -- 109 (1976).

\bibitem{Cuntz}
M. Cuntz, S. Elia, J.-P. Labb\'e, 
Congruence normality of simplicial hyperplane arrangements via oriented matroids. \textit{Ann. Comb.} \textbf{26(1)}: 1 -- 85 (2022).

\bibitem{DIPS} A. Dimca, G. Ilardi, P. Pokora,  G. Sticlaru, Construction of free curves by adding lines to a given curve. \textit{Results Math.} \textbf{79}: Art. Id 11 - 31 pages (2024). 

\bibitem{DimPok2023} 
A. Dimca and P. Pokora, Maximizing Curves Viewed as Free Curves. \textit{Int. Math. Res. Not. IMRN} \textbf{22}:  19156 -- 19183 (2023).

\bibitem{DimSer}
A. Dimca and E. Sernesi, Syzygies and logarithmic vector fields along plane curves. (Syzygies et champs de vecteurs logarithmiques le long de courbes planes.) \textit{J. Éc. Polytech., Math.} \textbf{1}: 247 -- 267 (2014).

\bibitem{duP}
A. A. Du Plessis and C. T. C. Wall, Application of the theory of the discriminant to highly singular plane curves. \textit{Math. Proc. Camb. Philos. Soc.} \textbf{126(2)}: 259 -- 266 (1999).

\bibitem{Singular} W.~Decker, G.-M. Greuel, G.~Pfister, and H.~Sch\"onemann, Singular {4-1-1} --- {A} computer algebra system for polynomial computations. \url{http://www.singular.uni-kl.de} (2018).

\bibitem{Persson}
U. Persson, Horikawa surfaces with maximal Picard numbers. \textit{Math. Ann.} \textbf{259}: 287 -- 312 (1982).

\bibitem{PN}
E. Paemurru and N. Viswanathan, Log canonical thresholds of high multiplicity reduced plane curves. \textit{Glasg. Math. J.} (accepted, in press), pp. 1 -- 14 (2025).

\bibitem{PP2024}
P. Pokora, Freeness of arrangements of lines and one conic with ordinary quasi-homogeneous singularities. \textit{Taiwanese Math. Journal}. Advance Publication 1 -- 16, 2024, \url{https://doi.org/10.11650/tjm/241105}.
\bibitem{Saito}
K. Saito, Theory of logarithmic differential forms and logarithmic vector fields. \textit{J. Fac. Sci., Univ. Tokyo, Sect. I A} \textbf{27}: 265 -- 291 (1980).
\end{thebibliography}
\end{document}